\documentclass{amsart}
\usepackage{amssymb,amsmath,amsthm,latexsym}

  \theoremstyle{definition}
  \newtheorem{definition}{Definition}[section]
  \newtheorem{notation}[definition]{Notation}
  
  \newtheorem{remark}[definition]{Remark}

  \theoremstyle{plain}
  \newtheorem{lemma}[definition]{Lemma}
  \newtheorem{proposition}[definition]{Proposition}
  \newtheorem{theorem}[definition]{Theorem}
  \newtheorem{corollary}[definition]{Corollary}
  \newtheorem*{conjecture}{Conjecture}

\begin{document}

\title[Universal associative envelopes]
{Universal associative envelopes of $(n{+}1)$-dimensional $n$-Lie algebras}

\author[Bremner]{Murray R. Bremner}

\address{Department of Mathematics and Statistics, University of Saskatchewan,
Canada}

\email{bremner@math.usask.ca}

\author[Elgendy]{Hader A. Elgendy}

\address{Department of Mathematics and Statistics, University of Saskatchewan,
Canada}

\email{hae431@mail.usask.ca}

\date{\textit{\today}}

\begin{abstract}
For $n$ even, we prove Pozhidaev's conjecture on the existence of associative
enveloping algebras for simple $n$-Lie algebras. More generally, for $n$ even
and any $(n{+}1)$-dimensional $n$-Lie algebra $L$, we construct a universal
associative enveloping algebra $U(L)$ and show that the natural map $L \to
U(L)$ is injective. We use noncommutative Gr\"obner bases to present $U(L)$ as
a quotient of the free associative algebra on a basis of $L$ and to obtain a
monomial basis of $U(L)$. In the last section, we provide computational
evidence that the construction of $U(L)$ is much more difficult for $n$ odd.
\end{abstract}



\maketitle


\section{Introduction}

Filippov \cite{Filippov} in 1985 introduced $n$-Lie algebras and classified the
$(n{+}1)$-dimen\-sional $n$-Lie algebras over an algebraically closed field of
characteristic 0.

\begin{definition} \cite{Filippov}
An \emph{$n$-Lie algebra} is a vector space $L$ over a field $F$ of
characteristic $\ne 2$ with a multilinear operation $[x_1, x_2, \dots, x_n]$
satisfying the \emph{alternating (or anticommutative) identity} and the
\emph{generalized Jacobi (or derivation) identity}:
  \begin{align*}
  [ x_1, x_2, \dots, x_n ]
  &=
  \epsilon(\sigma)
  [ x_{\sigma(1)}, x_{\sigma(2)}, \dots, x_{\sigma(n)} ]
  \quad
  (\sigma \in S_n),
  \\
  [ [ x_1, \dots, x_n ], y_2, \dots, y_n ]
  &=
  \sum^n_{i=1}
  [ x_1, \dots, [ x_i, y_2, \dots, y_n ], \dots, x_n ].
  \end{align*}
\end{definition}

For $n = 2$ we obtain the definition of a Lie algebra, but for $n \ge 3$ the
structure of $n$-Lie algebras is quite different. In particular, Ling
\cite{Ling} showed that for each $n \ge 3$ there exists up to isomorphism a
unique simple finite-dimensional $n$-Lie algebra over an algebraically closed
field of characterstic 0.

\begin{definition} \label{simpledefinition} \cite{Filippov}
Let $n \ge 3$ and let $F$ be a field of characteristic $\ne 2$. Let $L_{n+1}$
be the $(n{+}1)$-dimensional $n$-Lie algebra over $F$ with basis $e_1, \dots,
e_{n+1}$ such that
  \[
  [ e_1, \dots, \widehat{e}_i, \dots, e_{n+1} ]
  =
  (-1)^{n+i+1} e_i,
  \quad
  (1 \le i \le n{+}1);
  \]
$\widehat{e}_i$ means that $e_i$ is omitted.
Filippov \cite[Theorem 4]{Filippov} shows that $L_{n+1}$ is simple.
\end{definition}

Alternating $n$-ary structures have attracted attention in theoretical physics
during the last few decades. In particular, the important recent work by Bagger
and Lambert \cite{BaggerLambert} and Gustavsson \cite{Gustavsson} attempts to
describe an effective action for the low energy dynamics of coincident
M2-branes. For a very recent comprehensive survey on the physical applications
of $n$-ary algebras, see de Azc\'arraga and Izquierdo
\cite{deAzcarragaIzquierdo}.

The Poincar\'e-Birkhoff-Witt (PBW) theorem is an important tool in the
representation theory of Lie algebras. It provides a basis for the universal
associative enveloping algebra of any Lie algebra over any field, and allows us
to make calculations in these noncommutative algebras. Bergman \cite{Bergman}
in 1978 gave a new proof of the PBW theorem using noncommutative Gr\"obner
bases in the free associative algebra. This theory was used recently by Casas
et al.~\cite{CasasInsuaLadra} and Insua and Ladra \cite{InsuaLadra} to
construct universal enveloping algebras of Leibniz and $n$-Leibniz algebras.
Pozhidaev \cite{Pozhidaev} in 2003 showed that for $n \le 5$ the simple
finite-dimensional $n$-Lie algebra over an algebraically closed field of
characteristic $0$ can be embedded in an associative algebra, and made the
conjecture that such associative enveloping algebras exist for all $n$.

The aim of the present paper is to use noncommutative Gr\"obner bases to study
the universal associative enveloping algebras of $n$-Lie algebras and to
establish a generalization of the PBW theorem for $(n{+}1)$-dimensional $n$-Lie
algebras when $n$ is even. In Section \ref{prelimsection} we recall basic facts
about $n$-Lie algebras and noncommutative Gr\"obner bases. In Section
\ref{alternatingsection} we prove a theorem on the normal form of a composition
of ideal generators for universal associative  enveloping algebras of
$(n{+}1)$-dimensional alternating $n$-ary algebras. For $n$ even, this allows
us to construct a basis for $U(L)$ where $L$ is any $(n{+}1)$-dimensional
$n$-Lie algebra. In Section \ref{simplesection} we establish Pozhidaev's
conjecture for the simple $n$-Lie algebra when $n$ is even. In Section
\ref{nonsimplesection} we establish analogous results for the non-simple
$(n{+}1)$-dimensional $n$-Lie algebras. Finally, in Section \ref{oddsection} we
describe some calculations with the computer algebra system Maple which suggest
that extending these results to $n$ odd may be difficult.

Unless otherwise stated, we assume throughout that all vector spaces are over
an algebraically closed field $F$ of characteristic 0.


\section{Preliminaries} \label{prelimsection}

We first recall Filippov's classification of $(n{+}1)$-dimensional $n$-Lie
algebras. If $L$ is an $n$-Lie algebra then $L^1$ is its derived algebra and
$Z(L)$ is its center.

\begin{theorem} \label{classf} \cite{Filippov}
Let $n \ge 3$ and let $L$ be an $(n{+}1)$-dimensional $n$-Lie algebra with
basis $e_1, e_2, \dots, e_{n+1}$ over $F$. Up to isomorphism, exactly one of
the following cases holds; omitted brackets are assumed to be zero:
  \begin{enumerate}
  \item[(0)]
  If $\dim L^1 = 0$ then $L$ is the Abelian $n$-Lie algebra.
  \item[(1)]
  If $\dim L^1 = 1$ then we write $L^1 = F e_1$ and we have two cases:
    \begin{enumerate}
    \item If $L^1 \subseteq Z(L)$ then $[ e_2, \dots, e_{n+1} ] = e_1$.
    \item If $L^1 \nsubseteq Z(L)$ then $[ e_1, \dots, e_n ] = e_1$.
    \end{enumerate}
  \item[(2)]
  If $\dim L^1 = 2$ then we write $L^1 = F e_1 \oplus F e_2$ and we have two cases:
    \begin{enumerate}
    \item $[ e_2, \dots, e_{n+1} ] = e_1$ and $[ e_1, e_3, \dots, e_{n+1} ] = e_2$.
    \item $[ e_2, \dots, e_{n+1}] = e_1 + \beta e_2$ for $\beta \in F \setminus \{0\}$ and
    $[ e_1, e_3, \dots, e_{n+1} ] = e_2$.
    \end{enumerate}
  \item[($r$)]
  If $\dim L^1 = r$ for $3 \le r \le n+1$ then we write $L^1 = F e_1 \oplus \dots \oplus F e_r$ and we have
  $[ e_1, \dots, \widehat{e}_i, \dots, e_{n+1} ] = e_i$ for $1 \le i \le r$.
  \end{enumerate}
\end{theorem}

\begin{proof}
This is Filippov's classification \cite[Section 3]{Filippov} of
$(n{+}1)$-dimensional $n$-Lie algebras in the simplified version of Bai and
Song \cite[Theorem 3.1]{BaiSong}.
\end{proof}

We next recall the basic definitions and results in the theory of
noncommutative Gr\"obner bases in free associative algebras following de Graaf
\cite[Chapter 6]{deGraaf} which is based on the work of Bergman \cite{Bergman}.

\begin{definition}
Let $X = \{ x_1, \dots, x_{n+1} \}$ be a set of symbols with the total order
$x_i < x_j$ if and only if $i < j$. The \emph{free monoid} generated by $X$ is
the set $X^\ast$ of all (possibly empty) words $w = x_{i_1} \cdots x_{i_k}$ ($k
\ge 0$) with the (associative) operation of concatenation. For $w = x_{i_1}
\cdots x_{i_k} \in X^*$ the \emph{degree} is $\deg(w) = k$. The \emph{free
unital associative algebra} generated by $X$ is the vector space $F \langle X
\rangle$ with basis $X^*$ and multiplication extended bilinearly from
concatenation in $X^\ast$.
\end{definition}

\begin{definition}
Throughout this paper we use the \emph{degree-lexicographical} (\emph{deglex})
order $<$ on $X^\ast$ defined as follows: $u < v$ if and only if either ($i$)
$\deg(u) < \deg(v)$ or ($ii$) $\deg(u) = \deg(v)$ and $u = w x_i u'$, $v = w
x_j v'$ where $x_i < x_j$ ($w, u', v' \in X^\ast$). We say that $u \in X^*$ is
a \emph{factor} of $v \in X^*$ if there exist $w_1, w_2 \in X^*$ such that $w_1
u w_2 = v$. If $w_1$ (resp.~$w_2$) is empty then $u$ is a \emph{left}
(resp.~\emph{right}) factor of $v$.
\end{definition}

\begin{definition}
The \emph{support} of a noncommutative polynomial $f \in F\langle X \rangle$ is
the set of all monomials $w \in X^\ast$ that occur in $f$ with nonzero
coefficient. The \emph{leading monomial} of $f \in F\langle X\rangle$, denoted
$\mathrm{LM}(f)$, is the highest element of the support of $f$ with respect to
deglex order. If $I$ is any ideal of $F\langle X \rangle$ then the set of
\emph{normal words} modulo $I$ is defined by $N(I) = \{ \, u \in X^* \mid
\text{$u \ne \mathrm{LM}(f)$ for any $f \in I$} \, \}$. We write $C(I)$ for the
subspace of $F\langle X \rangle$ spanned by $N(I)$.
\end{definition}

\begin{proposition} \label{C(I)proposition}
If $I \subseteq F\langle  X\rangle$ is an ideal then $F\langle X \rangle = C(I)
\oplus I$.
\end{proposition}

\begin{proof}
de Graaf \cite[Proposition 6.1.1]{deGraaf}.
\end{proof}

\begin{definition}
Let $G \subseteq F\langle X \rangle$ be a subset generating an ideal $I
\subseteq F\langle X \rangle$. A noncommutative polynomial $f \in F\langle X
\rangle$ is in \emph{normal form modulo} $G$ if no monomial occurring in $f$
has a factor of the form $\mathrm{LM}(g)$ for any $g \in G$. For an algorithm
which calculates the normal form, see \cite[\S 6.1]{deGraaf}.
\end{definition}

\begin{definition}
If $I \subseteq F\langle X \rangle$ is an ideal then a subset $G \subseteq I$
is a \emph{Gr\"obner basis} of $I$ if for all $f \in I$ there is a $g \in G$
such that $\mathrm{LM}(g)$ is a factor of $\mathrm{LM}(f)$.
\end{definition}

\begin{definition}
A subset $G \subseteq F\langle X \rangle$ is \emph{self-reduced} if every $g
\in G$ is in normal form modulo $G \setminus \{g\}$ and every $g \in G$ is
\emph{monic}: the coefficient of $\mathrm{LM}(g)$ is 1. (This definition is
stronger than \cite[Definition 6.1.5]{deGraaf}.)
\end{definition}

\begin{definition}
Let $g, h \in F \langle X \rangle$ be two monic noncommutative polynomials.
Assume that $\mathrm{LM}(g)$ is not a factor of $\mathrm{LM}(h)$ and that
$\mathrm{LM}(h)$ is not a factor of $\mathrm{LM}(g)$. Let $u, v \in X^\ast$ be
such that
  \begin{enumerate}
  \item[($i$)] $\mathrm{LM}(g)\,u = v\,\mathrm{LM}(h)$,
  \item[($ii$)] $u$ is a proper right factor of $\mathrm{LM}(h)$,
  \item[($iii$)] $v$ is a proper left factor of $\mathrm{LM}(g)$.
  \end{enumerate}
In this case the element $g u - v h \in F \langle X \rangle$ is called a
\emph{composition} of $g$ and $h$.
\end{definition}

\begin{theorem} \label{di}
If $I \subseteq F\langle X \rangle$ is an ideal generated by a self-reduced set
$G$, then $G$ is a Gr\"obner basis of $I$ if and only if for all compositions
$f$ of the elements of $G$ the normal form of $f$ modulo $G$ is zero.
\end{theorem}

\begin{proof}
de Graaf \cite[Theorem 6.1.6, Corollary 6.1.8]{deGraaf}.
\end{proof}


\section{Universal associative envelopes of alternating $n$-ary algebras}
\label{alternatingsection}

Let $L$ be an $(n{+}1)$-dimensional $n$-ary algebra with an alternating product
which does not necessarily satisfy the generalized Jacobi identity. We are
primarily interested in $n$-Lie algebras but in this section we consider a more
general situation.

\begin{notation} \label{ydefinition}
Let $B = \{ e_1, e_2, \dots, e_{n+1} \}$ be an ordered basis of $L$. Consider
the bijection $\phi\colon B \to X = \{ x_1, x_2, \dots, x_{n+1} \}$ defined by
$\phi(e_i) = x_i$. We extend $\phi$ to a linear map $\phi\colon L\to F \langle
X \rangle$ and write $y_i = \phi( [ e_1, \dots, \widehat{e_i}, \dots, e_{n+1}
])$.
\end{notation}

\begin{definition}
Let $A$ be an associative algebra. On the underlying vector space of $A$ we
define a new operation, the $n$-ary \emph{alternating sum}:
  \[
  \mathrm{alt}( x_1, x_2, \dots, x_n )
  =
  \sum_{\sigma \in S_n}
  \epsilon(\sigma) \,
  x_{\sigma(1)} \, x_{\sigma(2)} \, \cdots \, x_{\sigma(n)}.
  \]
We write $A^-$ for the \emph{minus algebra}: the alternating $n$-ary algebra
obtained by replacing the associative product by the alternating sum.
\end{definition}

\begin{definition} \label{uaedefinition}
A \emph{universal associative envelope} of the alternating $n$-ary algebra $L$
consists of a unital associative algebra $U$ and a linear map $i\colon L \to U$
satisfying
  \[
  i([x_1, x_2,\dots, x_n])
  =
  \mathrm{alt}
  \big(
  i(x_1), i(x_2), \dots, i(x_n)
  \big)
  \quad
  (x_1, \dots, x_n \in L),
  \]
such that for any unital associative algebra $A$ and linear map $j\colon L \to
A$ satisfying the same equation with $j$ in place of $i$, there is a unique
homomorphism of unital associative algebras $\psi\colon U \to A$ such that
$\psi \circ i = j$.
\end{definition}

\begin{definition} \label{idealgenerators}
Consider the following elements of $F\langle X \rangle$ for $1 \le i \le n+1$:
  \[
  G_i =
  (-1)^{\lfloor n/2 \rfloor}
  \big(
  \mathrm{alt}
  (
  x_1, \dots, \widehat{x}_i, \dots, x_{n+1}
  )
  -
  y_i
  \big).
  \]
The factor $(-1)^{\lfloor n/2 \rfloor}$ ensures that $\mathrm{LM}(G_i) =
x_{n+1} \cdots \widehat{x}_i \cdots x_1$ has coefficient 1.
\end{definition}

\begin{notation}
Let $I \subseteq F\langle X \rangle$ be the ideal generated by $G_1, \dots,
G_{n+1}$. We write $U = F\langle X \rangle / I$ with surjection $\pi\colon
F\langle X \rangle \to U$ sending $f$ to $f + I$, and $i = \pi \circ \phi$ for
the natural map $i \colon L \to U$.
\end{notation}

\begin{lemma}
The unital associative algebra $U$ and the linear map $i$ form the universal
associative envelope of the alternating $n$-ary algebra $L$.
\end{lemma}

\begin{proof}
Similar to the case $n = 2$; see Humphreys \cite[\S 17.2]{Humphreys}.
\end{proof}

\begin{lemma}\label{com}
There is only one overlap among $\mathrm{LM}(G_1), \dots,
\mathrm{LM}(G_{n+1})$, namely $\mathrm{LM}(G_1) = x_{n+1} \cdots x_2$ and
$\mathrm{LM}(G_{n+1}) = x_n \cdots x_1$ have the common factor $x_n \cdots
x_2$. Hence there is only one composition among the generators: $G_1 x_1 -
x_{n+1} G_{n+1}$.
\end{lemma}

\begin{proof}
The subscripts in $\mathrm{LM}(G_i)$ are the sequence $n{+}1 > \cdots >
\widehat{i} > \cdots > 1$.
\end{proof}

\begin{theorem} \label{Nff}
The normal form of the composition $G_1 x_1 - x_{n+1} G_{n+1}$ is
  \[
  N
  =
  (-1)^n
  \sum_{i=1}^{n+1}
  (-1)^i
  \big(
  x_i G_i - (-1)^n G_i x_i
  \big)
  =
  (-1)^{n+1}
  \sum_{i=1}^{n+1} (-1)^i \big( x_i y_i - (-1)^n y_i x_i \big).
  \]
\end{theorem}

\begin{proof}
We first observe that $N$ can be rewritten as follows:
  \[
  (-1)^n
  \sum_{i=1}^{n+1}
  (-1)^i
  \big(
  x_i G_i - (-1)^n G_i x_i
  \big)
  =
  G_1 x_1 - x_{n+1} G_{n+1} + S + T,
  \]
where
  \[
  S =
  (-1)^n
  \sum_{i=2}^n
  (-1)^i
  \big(
  x_i G_i - (-1)^n G_i x_i
  \big),
  \quad
  T =
  -
  (-1)^n x_1 G_1 + (-1)^n G_{n+1} x_{n+1}.
  \]
To compute the normal form of $G_1 x_1 - x_{n+1} G_{n+1}$ using noncommutative
division with remainder, we perform two steps. First, we eliminate occurrences
of $\mathrm{LM}(G_1), \dots, \mathrm{LM}(G_{n+1})$ as factors in the monomials;
this corresponds to the sum $S$, and introduces new occurrences of
$\mathrm{LM}(G_1)$ and $\mathrm{LM}(G_{n+1})$. Second, we eliminate these last
two occurrences; this corresponds to the sum $T$. This shows that $N$ can be
obtained from $G_1 x_1 - x_{n+1} G_{n+1}$ by a sequence  of reductions modulo
the generators $G_1, \dots, G_{n+1}$ of the ideal $I$. Thus, in order to prove
that $N$ is the normal form of the composition, it remains to show that no
monomial occurring in $N$ has a factor equal to $\mathrm{LM}(G_i)$ for any $i =
1, \dots, n{+}1$.

For the following calculations, it is convenient to write
  \[
  G_i = (-1)^{\lfloor n/2 \rfloor}
  \Big(
  \sum_{\sigma \in S_n^{(i)}}
  \epsilon(\sigma)
  x_{\sigma(1)} \cdots \widehat{x}_{\sigma(i)} \cdots x_{\sigma(n+1)}
  -
  y_i
  \Big),
  \]
where $S_n^{(i)} \cong S_n$ is the symmetric group on $\{ 1, \dots,
\widehat{i}, \dots, n{+}1 \}$. To simplify the signs, we factor out $(-1)^n
(-1)^{\lfloor n/2 \rfloor}$ from the entire calculation. Thus we consider the
following simplified versions of $N$ and the ideal generators $G_i$:
  \[
  N = \sum_{i=1}^{n+1}
  (-1)^i
  \big(
  x_i G_i - (-1)^n G_i x_i
  \big),
  \quad
  G_i =
  \sum_{\sigma \in S_n^{(i)}}
  \epsilon(\sigma)
  x_{\sigma(1)} \cdots \widehat{x}_{\sigma(i)} \cdots x_{\sigma(n+1)}
  -
  y_i.
  \]
We rewrite $x_i G_i$ and $G_i x_i$ as follows:
  \begin{align}
  \label{xiGi}
  x_i G_i
  &=
  \sum_{\tiny \begin{array}{c} \tau \in S_{n+1} \\ \tau(1) = i \end{array}}
  (-1)^{i-1}
  \epsilon(\tau)
  x_i x_{\tau(2)} \cdots x_{\tau(n+1)}
  -
  x_i y_i,
  \\
  \label{Gixi}
  G_i x_i
  &=
  \sum_{\tiny \begin{array}{c} \tau \in S_{n+1} \\ \tau(n{+}1) = i \end{array}}
  (-1)^{n+1-i}
  \epsilon(\tau)
  x_{\tau(1)} \cdots x_{\tau(n)} x_i
  -
  y_i x_i.
  \end{align}
In $x_i G_i$, the symbol $x_i$ has moved left past $i{-}1$ symbols, so
$\epsilon(\sigma) = (-1)^{i-1} \epsilon(\tau)$. In $G_i x_i$, the symbol $x_i$
has moved right past $n{+}1{-}i$ symbols, so $\epsilon(\sigma)= (-1)^{n+1-i}
\epsilon(\tau)$. Therefore
  \begin{align*}
  &
  x_i G_i - (-1)^n G_i x_i
  =
  -
  \big( x_i y_i - (-1)^n y_i x_i \big)
  \\
  &
  +
  \!\!\!\!
  \sum_{\tiny \begin{array}{c} \tau \in S_{n+1} \\ \tau(1) = i \end{array}}
  \!\!\!\!
  (-1)^{i-1}
  \epsilon(\tau)
  x_i x_{\tau(2)} \cdots x_{\tau(n+1)}
  -
  \!\!\!\!
  \sum_{\tiny \begin{array}{c} \tau \in S_{n+1} \\ \tau(n{+}1) = i \end{array}}
  \!\!\!\!
  (-1)^{i-1}
  \epsilon(\tau)
  x_{\tau(1)} \cdots x_{\tau(n)} x_i.
  \end{align*}
From this we obtain
  \begin{align*}
  &
  (-1)^i \big( x_i G_i - (-1)^n G_i x_i \big)
  =
  - (-1)^i
  \big(
  x_i y_i - (-1)^n y_i x_i
  \big)
  \\
  &
  -
  \sum_{\tiny \begin{array}{c} \tau \in S_{n+1} \\ \tau(1) = i \end{array}}
  \epsilon(\tau)
  x_i x_{\tau(2)} \cdots x_{\tau(n+1)}
  +
  \sum_{\tiny \begin{array}{c} \tau \in S_{n+1} \\ \tau(n{+}1) = i \end{array}}
  \epsilon(\tau)
  x_{\tau(1)} \cdots x_{\tau(n)} x_i.
  \end{align*}
Summing over $i = 1, \dots, n{+}1$ gives
  \begin{align*}
  &
  \sum_{i=1}^{n+1}
  (-1)^i \big( x_i G_i - (-1)^n G_i x_i \big)
  =
  Q + R,
  \quad
  Q
  =
  - \sum_{i=1}^{n+1} (-1)^i \big( x_i y_i - (-1)^n y_i x_i \big),
  \\
  &
  R
  =
  {}
  -
  \sum_{i=1}^{n+1}
  \!\!\!\!
  \sum_{\tiny \begin{array}{c} \tau \in S_{n+1} \\ \tau(1) = i \end{array}}
  \!\!\!\!
  \epsilon(\tau)
  x_i x_{\tau(2)} \cdots x_{\tau(n+1)}
  +
  \sum_{i=1}^{n+1}
  \!\!\!\!
  \sum_{\tiny \begin{array}{c} \tau \in S_{n+1} \\ \tau(n{+}1) = i \end{array}}
  \!\!\!\!
  \epsilon(\tau)
  x_{\tau(1)} \cdots x_{\tau(n)} x_i.
  \end{align*}
It remains to show that $R = 0$. In the first (respectively second) double sum,
we separate terms according to the last (respectively first) symbol in each
monomial:
  \begin{align*}
  R
  &=
  -
  \sum_{i=1}^{n+1}
  \sum_{\tiny \begin{array}{c} j=1 \\ j \ne i \end{array}}^{n+1}
  \sum_{\tiny \begin{array}{c} \tau \in S_{n+1} \\ \tau(1) = i \\ \tau(n{+}1) = j \end{array}}
  \epsilon(\tau)
  x_i x_{\tau(2)} \cdots x_{\tau(n)} x_j
  \\
  &\qquad
  +
  \sum_{i=1}^{n+1}
  \sum_{\tiny \begin{array}{c} j=1 \\ j \ne i \end{array}}^{n+1}
  \sum_{\tiny \begin{array}{c} \tau \in S_{n+1} \\ \tau(1) = j \\ \tau(n{+}1) = i \end{array}}
  \epsilon(\tau)
  x_j x_{\tau(2)} \cdots x_{\tau(n)} x_i
  =
  0,
  \end{align*}
since both sums are over all pairs $(i,j)$ with $1 \le i \ne j \le n{+}1$.
\end{proof}

\begin{remark}
For $n$ even (respectively odd) the terms of $N$ can be written as Lie brackets
(respectively Jordan products):
  \[
  x_i G_i - (-1)^n G_i x_i = [ x_i, G_i ] \;\; \text{($n$ even)},
  \quad
  x_i G_i - (-1)^n G_i x_i = x_i \circ G_i \;\; \text{($n$ odd)}.
  \]
\end{remark}


\section{Pozhidaev's conjecture for simple $n$-Lie algebras ($n$ even)}
\label{simplesection}

In the rest of this paper, we assume that $L$ is an $n$-Lie algebra. Pozhidaev
\cite{Pozhidaev} considered the problem whether there exists an embedding of an
arbitrary $n$-Lie algebra into an associative algebra, and made the following
conjecture:

\begin{conjecture}
For any reductive finite-dimensional $n$-Lie algebra $L$ over an algebraically
closed field of characteristic 0 there exists an associative algebra $A$ such
that $L$ is isomorphic to a subalgebra of $A^-$.
\end{conjecture}

By the work of Ling \cite{Ling} it is known that any reductive
finite-dimensional $n$-Lie algebra over an algebraically closed field of
characteristic 0 decomposes into the direct sum of an Abelian ideal and several
copies of a simple ideal isomorphic to the simple $(n{+}1)$-dimensional $n$-Lie
algebra $L_{n+1}$. Hence the main problem is to prove that $L_{n+1}$ can be
embedded into an associative algebra.

\begin{theorem}\label{comproof}
Let $n \ge 3$ and let $F$ be a field of characteristic $\ne 2$. Let $L$ be the
simple $n$-Lie algebra $L_{n+1}$ over $F$ from Definition
\ref{simpledefinition}. The generators $\{ G_1, \dots, G_{n+1} \}$ of
Definition \ref{idealgenerators} form a Gr\"obner basis for the ideal $I =
\langle G_1, \dots, G_{n+1} \rangle$ in the free associative algebra $F\langle
x_1, \dots, x_{n+1} \rangle$ if and only if $n$ is even.
\end{theorem}

\begin{proof}
The structure constants for $L_{n+1}$ give $y_i = (-1)^{n+i+1} x_i$ and so
  \[
  G_i =
  (-1)^{\lfloor n/2 \rfloor}
  \big(
  \mathrm{alt}
  (
  x_1, \dots, \widehat{x}_i, \dots, x_{n+1}
  )
  +
  (-1)^{n+i} x_i
  \big).
  \]
By Theorem \ref{Nff} the normal form of the single composition of these
generators is
  \[
  N
  =
  \sum_{i=1}^{n+1}
  \big( 1 - (-1)^n \big)
  x_i^2
  =
  \begin{cases}
  0 &\text{if $n$ is even}, \\
  2 \sum_{i=1}^{n+1} x_i^2 &\text{if $n$ is odd}.
  \end{cases}
  \]
Since $\mathrm{char}\,F \ne 2$ we have $N = 0$ if and only if $n$ is even, and
by Theorem \ref{di} this is equivalent to $\{ G_1, \dots, G_{n+1} \}$ being a
Gr\"{o}bner basis.
\end{proof}

\begin{corollary} \label{u}
Let $n \ge 4$ be even and let $F$ be a field of characteristic $\ne 2$. Let $L$
be the simple $n$-Lie algebra $L_{n+1}$ over $F$. The universal associative
enveloping algebra $U(L)$ is infinite-dimensional, and a basis consists of the
monomials which do not contain any factor of the form $x_{i_1} \cdots x_{i_n}$
with $i_1 > \cdots > i_n$.
\end{corollary}

\begin{proof}
Since $\{ G_1, \dots, G_{n+1} \}$ is Gr\"obner basis, Proposition
\ref{C(I)proposition} shows that the normal words of $F\langle X\rangle$ modulo
$I$, or equivalently the coset representatives for $U(L) = F\langle X \rangle /
I$, are those that do not contain any $\mathrm{LM}(G_i)$ as a factor.
\end{proof}

\begin{corollary} \label{simpleinjective}
Let $n \ge 4$ be even and let $F$ be a field of characteristic $\ne 2$. For the
simple $n$-Lie algebra $L_{n+1}$ the natural map $i\colon L_{n+1} \to
U(L_{n+1})$ is injective.
\end{corollary}

\begin{proof}
The intersection of $I = \langle G_1, \dots, G_{n+1} \rangle$ with
$\mathrm{span}( x_1, \dots, x_{n+1} )$ is 0, and hence the cosets of the $x_i$
are linearly independent in $U(L_{n+1})$.
\end{proof}

We now obtain a proof of Pozhidaev's conjecture \cite{Pozhidaev} in the case of
$n$ even.

\begin{corollary}
Let $n \ge 4$ be even and let $F$ be a field of characteristic $\ne 2$. There
exists an associative algebra $A$ such that the simple $n$-Lie algebra
$L_{n+1}$ is isomorphic to a subalgebra of $A^-$.
\end{corollary}

\begin{proof}
Take $A = U(L)$ and apply Corollary \ref{simpleinjective}.
\end{proof}

We also obtain the following new proof of Pozhidaev's Corollary 2.1
\cite{Pozhidaev}.

\begin{corollary}
Let $n \ge 3$ be odd, let $F$ be a field of characteristic $\ne 2$, and let $L$
be the simple $n$-Lie algebra $L_{n+1}$. If $A$ is an associative algebra and
$j\colon L \to A^-$ is a homomorphism of alternating $n$-ary algebras, then
$j(e_1)^2 + \cdots j(e_{n+1})^2 = 0$.
\end{corollary}

\begin{proof}
The proof of Theorem \ref{comproof} shows that $x_1^2 + \cdots + x_{n+1}^2 = 0$
in $U(L)$, and so the claim follows from the universal property of $U(L)$.
\end{proof}

For $n$ odd, finding a Gr\"obner basis of $I = \langle G_1, \dots, G_{n+1}
\rangle$ for the simple $n$-Lie algebra $L_{n+1}$ seems to be much more
difficult; see the calculations in Section \ref{oddsection}.


\section{The non-simple $n$-Lie algebras ($n$ even)} \label{nonsimplesection}

We now consider the other $(n{+}1)$-dimensional $n$-Lie algebras in the
classification of Theorem \ref{classf}. We divide these non-simple algebras
into three cases depending on the complexity of the resulting Gr\"obner basis.

\subsection{Case 1}

This includes cases (0), (1a), (2a) and ($r$) of Theorem \ref{classf}.

\begin{theorem}
Let $n \ge 4$ be even, let $F$ be an algebraically closed field of
characteristic 0, and let $L$ be an $(n{+}1)$-dimensional $n$-Lie algebra from
Theorem \ref{classf}. In the following four cases, the original ideal
generators $\{ G_1, \dots, G_{n+1} \}$ of Definition \ref{idealgenerators} are
a Gr\"obner basis for the ideal $I = \langle G_1, \dots, G_{n+1} \rangle
\subseteq F\langle X \rangle$:
  \begin{enumerate}
  \item[(0)] $L^1 = \{0\}$: $L$ is the Abelian $n$-Lie algebra.
  \item[(1)] \emph{(a)} $L^1 = F e_1$ where $[ e_2, \dots, e_{n+1} ] =
      e_1$.
  \item[(2)] \emph{(a)} $L^1 = F e_1 \oplus F e_2$ where $[ e_2, \dots,
      e_{n+1} ] = e_1$ and $[ e_1, e_3, \dots, e_{n+1} ] = e_2$.
  \item[($r$)] $L^1 = F e_1 \oplus \dots \oplus F e_r$ ($3 \le r \le n$)
      where $[ e_1, \dots, \widehat{e}_i, \dots, e_{n+1} ] = e_i$ for $1
      \le i \le r$.
  \end{enumerate}
\end{theorem}

\begin{proof}
In each case we verify that the normal form $N$ of the unique composition of
the original ideal generators is equal to 0. This is trivial in case (0). In
case (1)(a), Theorem \ref{Nff} gives $N = x_1^2 - x_1^2 = 0$. In case (2)(a),
we get $N = ( x_1^2 - x_1^2 ) - ( x_2^2 - x_2^2 ) = 0$. In case ($r$), we get
$N = - \sum_{i=1}^r (-1)^i ( x_i^2 - x_i^2 ) = 0$. We note that in all these
cases, either $y_i = x_i$ or $y_i = 0$ for $i = 1, \dots, n{+}1$ .
\end{proof}

\subsection{Case 2}

This is case (1b) of Theorem \ref{classf}: $L^1 = F e_1$ where $[ e_1, \dots,
e_n ] = e_1$. The original ideal generators are
  \begin{align*}
  G_i
  &=
  (-1)^{\lfloor n/2 \rfloor}
  \mathrm{alt}
  (
  x_1, \dots, \widehat{x}_i, \dots, x_{n+1}
  )
  \quad
  (1 \le i \le n),
  \\
  G_{n+1}
  &=
  (-1)^{\lfloor n/2 \rfloor}
  \big(
  \mathrm{alt}
  (
  x_1, \dots, x_n
  )
  -
  x_1
  \big).
  \end{align*}

\begin{lemma}
The composition $G_1 x_1- x_{n+1} G_{n+1}$ has normal form
  \[
  N = x_{n+1} x_1 - x_1 x_{n+1}.
  \]
\end{lemma}

\begin{proof}
This follows directly from Theorem \ref{Nff}.
\end{proof}

We must include $N$ as a new generator and modify the original generators by
replacing them by their normal forms modulo $N$.

\begin{notation}
For $i = 2, \dots, n$ we write $T_n^{(i)}$ for the set of all permutations of
$\{ 1$, \dots, $\widehat{i}$, \dots, $n{+}1 \}$ in which 1 and $n{+}1$ do not
appear consecutively. We consider the following corresponding elements of
$F\langle X \rangle$:
  \[
  H_i
  =
  (-1)^{\lfloor n/2 \rfloor}
  \sum_{\sigma \in T_n^{(i)}}
  \epsilon(\sigma) \,
  x_{\sigma(1)} \, x_{\sigma(2)} \, \cdots \, x_{\sigma(n)}
  \quad
  (2 \le i \le n).
  \]
\end{notation}

\begin{theorem}
Let $n \ge 4$ be even and let $F$ be any field. Let $L$ be the
$(n{+}1)$-dimensional $n$-Lie algebra with structure constants $[ e_1, \dots,
e_n ] = e_1$. A Gr\"obner basis for the ideal $I = \langle G_1, \dots, G_{n+1}
\rangle \subseteq F\langle X \rangle$ consists of the elements
  \[
  \{ \, G_1, \, H_2, \, \dots, \, H_n, \, G_{n+1}, \, N \, \}.
  \]
\end{theorem}

\begin{proof}
We have $\mathrm{LM}(N) = x_{n+1} x_1$ and obviously this never occurs as a
factor of any monomial in $G_1$ or $G_{n+1}$. If $x_{n+1} x_1$ is a factor of a
term $\epsilon w = \pm u x_{n+1} x_1 v$ occurring in $G_i$ for some $i = 2,
\dots, n$, then we reduce $w$ using $N$. This simply means that we replace
$\epsilon w$ by $\epsilon w' = \pm u x_1 x_{n+1} v$. But since $G_i$ is an
alternating sum, the term $-\epsilon w'$ appears in $G_i$, and the terms
$\epsilon w'$ and $-\epsilon w'$ cancel. The remaining terms in $G_i$
correspond to the permutations in $T_n^{(i)}$ and so we obtain the new
generators $H_2, \dots, H_n$. No further reductions are possible in the set of
generators: the set $\{ G_1, H_2, \dots, H_n, G_{n+1}, N \}$ is self-reduced.
The leading monomials of the generators $G_1, H_2, \dots, H_n, G_{n+1}$ have
strictly decreasing subscripts, and hence never have $x_1$ as the first symbol
or $x_{n+1}$ as the last symbol; it follows that no further compositions with
$N$ are possible. Hence we now have a Gr\"obner basis for the ideal $I$.
\end{proof}

\subsection{Case 3}
This is case (2b) of Theorem \ref{classf}: $L^1 = F e_1 \oplus F e_2$ where
  \[
  [ e_2, \dots, e_{n+1} ] = e_1 + \beta e_2 \; (\beta \ne 0),
  \qquad
  [ e_1, e_3, \dots, e_{n+1} ] = e_2.
  \]
The original ideal generators are
  \begin{align*}
  G_1
  &=
  (-1)^{\lfloor n/2 \rfloor}
  \big(
  \mathrm{alt}
  (
  x_2, \dots, x_{n+1}
  )
  -
  ( x_1 +\beta x_2 )
  \big),
  \\
  G_2
  &=
  (-1)^{\lfloor n/2 \rfloor}
  \big(
  \mathrm{alt}
  (
  x_1, x_3, \dots, x_{n+1}
  )
  -
  x_2
  \big),
  \\
  G_i
  &=
  (-1)^{\lfloor n/2 \rfloor}
  \mathrm{alt}
  (
  x_1, \dots, \widehat{x}_i, \dots, x_{n+1}
  )
  \quad
  (3 \le i \le n{+}1).
  \end{align*}

\begin{lemma}
The composition $G_1 x_1- x_{n+1} G_{n+1}$ has normal form
  \[
  N = x_2 x_1 - x_1 x_2.
  \]
\end{lemma}

\begin{proof}
This follows directly from Theorem \ref{Nff} since $\beta \ne 0$.
\end{proof}

We must include $N$ as a new generator and modify the original generators by
replacing them by their normal forms modulo $N$.

\begin{notation} \label{VKnotation}
We write $V_n^{(i)}$ ($i \ne 1, 2$) for the set of permutations of $\{ 1,
\dots, \widehat{i}$, $\dots$, $n{+}1 \}$ in which 1 and 2 do not appear
consecutively. We consider the corresponding elements of $F\langle X \rangle$:
  \[
  K_i
  =
  -(-1)^{\lfloor n/2 \rfloor}
  \sum_{\sigma \in V_n^{(i)}}
  \epsilon(\sigma) \,
  x_{\sigma(1)} \, x_{\sigma(2)} \, \cdots \, x_{\sigma(n)}
  \quad
  (3 \le i \le n{+}1).
  \]
The extra minus sign appears because $\mathrm{LM}(K_i)$ differs by a
transposition from $\mathrm{LM}(G_i)$: the leading monomial of $K_i$ is
  \[
  x_{n+1} \cdots x_5 x_2 x_4 x_1 \; (i = 3),
  \qquad
  x_{n+1} \cdots \widehat{x}_i \cdots x_4 x_2 x_3 x_1 \; (i \ge 4).
  \]
We write $V_{n+1}$ for the subset of $S_{n+1}$ in which 1 and 2 do not appear
consecutively.
\end{notation}

\begin{theorem}
Let $n \ge 4$ be even and let $F$ be any field. Let $L$ be the
$(n{+}1)$-dimensional $n$-Lie algebra with structure constants
  \[
  [ e_2, \dots, e_{n+1} ] = e_1 + \beta e_2 \; (\beta \ne 0),
  \qquad
  [ e_1, e_3, \dots, e_{n+1} ] = e_2.
  \]
A Gr\"obner basis for the ideal $I = \langle G_1, \dots, G_{n+1} \rangle
\subseteq F\langle X \rangle$ consists of the elements
  \[
  \{ \, G_1, G_2, K_3, \dots, K_{n+1}, N \, \}.
  \]
\end{theorem}

\begin{proof}
We first use $N$ to reduce the original generators $G_1, \dots, G_{n+1}$.
Clearly $G_1$ and $G_2$ do not change, since $G_1$ (resp.~$G_2$) does not
contain $x_1$ (resp.~$x_2$). The monomials in $G_3, \dots, G_{n+1}$ of the form
$\cdots x_2 x_1 \cdots$ reduce to $\cdots x_1 x_2 \cdots$; hence all the
monomials containing $x_2 x_1$ and $x_1 x_2$ cancel, and $G_3, \dots, G_{n+1}$
reduce to $K_3, \dots, K_{n+1}$. It is easy to check that $G_1$, $G_2$, $K_3$,
$\dots$, $K_{n+1}$, $N$ have only one overlap among their leading monomials:
$\mathrm{LM}(G_1) = x_{n+1} \cdots x_2$, $\mathrm{LM}(N) = x_2 x_1$. Hence
there is a single new composition,
  \[
  P = G_1 x_1 - x_{n+1} x_n \dots x_3 N.
  \]
To complete the proof, it suffices to show that the normal form of $P$ is 0.

Following the proof of Theorem \ref{Nff}, we first eliminate from $P$ all
occurrences of the leading monomials of $G_2$, $K_3$, $\dots$, $K_{n+1}$, $N$.
This gives $P + Q$ where
  \[
  Q
  =
  -
  G_2 x_2
  +
  x_2 G_2
  +
  \sum_{i=3}^{n+1}
  (-1)^{i+1}
  x_i K_i
  +
  (-1)^{n/2}
  \Bigg[
  \beta
  N
  -
  \!\!\!\!\!\!\!\!
  \sum_{\tiny
  \begin{array}{c} \tau \in S_{n+1} \\ \tau(n) = 2 \\ \tau(n{+}1)=1 \end{array}}
  \!\!\!\!\!\!\!\!
  {\epsilon(\tau)}
  x_{\tau(1)} \cdots x_{\tau(n-1)} N
  \Bigg].
  \]
We next eliminate from $P + Q$ all occurrences of the leading monomials of
$G_1$, $K_3$, $\dots$, $K_{n+1}$, $N$. This gives $P + Q + R$ where
  \[
  R
  =
  -
  x_1 G_1
  -
  \sum_{i=3}^{n+1}
  (-1)^{i+1}
  K_i x_i
  +
  (-1)^{n/2}
  \!\!\!\!\!\!\!\!
  \sum_{\tiny
  \begin{array}{c} \tau \in S_{n+1} \\ \tau(1) = 2 \\ \tau(2) = 1 \end{array}}
  \!\!\!\!\!\!\!\!
  \epsilon(\tau)
  N x_{\tau(3)} \cdots x_{\tau(n+1)}.
  \]
This shows that $P$ reduces to $M = P + Q + R$.  It remains to show that $M =
0$.

Combining the terms in $P$, $Q$, $R$ we obtain $M = A + B + C$ where
  \allowdisplaybreaks
  \begin{align*}
  A
  &=
  \sum_{i=1}^2 (-1)^{i+1} ( G_i x_i - x_i G_i ),
  \qquad
  B
  =
  \sum_{i=3}^{n+1} (-1)^{i+1} ( x_i K_i - K_i x_i ),
  \\
  C
  &=
  (-1)^{n/2}
  \Bigg[
  \beta N +
  \!\!\!\!\!\!\!
  \sum_{\tiny\begin{array}{c} \tau \in S_{n+1} \\ \tau(1) = 2 \\ \tau(2) = 1 \end{array}}
  \!\!\!\!\!\!\!
  \epsilon(\tau)
  N x_{\tau(3)} \cdots x_{\tau(n+1)}
  -
  \!\!\!\!\!\!\!
  \sum_{\tiny\begin{array}{c} \tau \in S_{n+1} \\ \tau(n) = 2 \\ \tau(n{+}1)=1 \end{array}}
  \!\!\!\!\!\!\!
  \epsilon(\tau)
  x_{\tau(1)} \cdots x_{\tau(n-1)} N
  \Bigg].
  \end{align*}
We factor out $(-1)^{n/2}$ from the following calculation to simplify the signs.

Using the definitions of the ideal generators, we rewrite $A$ as follows:
  \allowdisplaybreaks
  \begin{align*}
  A
  &=
  \sum_{i=1}^2
  \Bigg[
  \sum_{\tiny \begin{array}{c} \tau \in S_{n+1} \\ \tau(n{+}1) = i \end{array}}
  \!\!\!\!\!\!
  \epsilon(\tau)
  x_{\tau(1)} \cdots x_{\tau(n)} x_i
  -
  \!\!\!\!\!\!
  \sum_{\tiny \begin{array}{c} \tau \in S_{n+1} \\ \tau(1) = i \end{array}}
  \!\!\!\!\!\!
  \epsilon(\tau)
  x_i x_{\tau(2)} \cdots x_{\tau(n+1)}
  \Bigg]
  \\
  &\quad
  - \beta (x_2 x_1 - x_1 x_2).
  \end{align*}
The signs $(-1)^{i+1}$ cancel using equations \eqref{xiGi} and \eqref{Gixi}. We
now separate the monomials which either begin or end with either $x_1 x_2$ or
$x_2 x_1$:
  \begin{align*}
  A
  &=
  \sum_{i=1}^2
  \Bigg[
  \sum_{\tiny \begin{array}{c} \tau \in V_{n+1} \\ \tau(n{+}1) = i \end{array}}
  \!\!\!\!\!\!
  \epsilon(\tau)
  x_{\tau(1)} \cdots x_{\tau(n)} x_i
  -
  \!\!\!\!\!\!
  \sum_{\tiny \begin{array}{c} \tau \in V_{n+1} \\ \tau(1) = i \end{array}}
  \!\!\!\!\!\!
  \epsilon(\tau)
  x_i x_{\tau(2)} \cdots x_{\tau(n+1)}
  \Bigg]
  \\
  &\quad
  +
  \!\!\!\!\!\!
  \sum_{\tiny \begin{array}{c} \tau \in S_{n+1} \\ \tau(n) = 2 \\ \tau(n{+}1)=1 \end{array}}
  \!\!\!\!\!\!
  {\epsilon(\tau)}
  x_{\tau(1)} \cdots x_{\tau(n-1)}(x_2x_1-x_1x_2)
  \\
  &\quad
  -
  \!\!\!\!\!\!
  \sum_{\tiny \begin{array}{c} \tau \in S_{n+1} \\ \tau(1) = 2 \\ \tau(2) = 1 \end{array}}
  \!\!\!\!\!\!
  \epsilon(\tau)
  (x_2x_1 - x_1 x_2) x_{\tau(3)} \cdots x_{\tau(n+1)}
  -
  \beta (x_2 x_1 - x_1 x_2).
  \end{align*}

Similarly, we obtain
  \[
  B
  =
  \sum_{i=3}^{n+1}
  \Bigg[
  \sum_{\tiny \begin{array}{c} \tau \in V_{n+1} \\ \tau(n{+}1) = i \end{array}}
  \!\!\!\!\!\!
  \epsilon(\tau)
  x_{\tau(1)} \cdots x_{\tau(n)} x_i
  -
  \!\!\!\!\!\!
  \sum_{\tiny \begin{array}{c} \tau \in V_{n+1} \\ \tau(1) = i \end{array}}
  \!\!\!\!\!\!
  \epsilon(\tau)
  x_i x_{\tau(2)} \cdots x_{\tau(n+1)}
  \Bigg],
  \]
using the same relation between $\epsilon(\sigma)$ and $\epsilon(\tau)$ as in
equations \eqref{xiGi} and \eqref{Gixi}.

Since $N = x_2 x_1 - x_1 x_2$ we obtain
  \begin{align*}
  C
  &=
  \sum_{\tiny \begin{array}{c} \tau \in S_{n+1} \\ \tau(1) = 2 \\ \tau(2) = 1 \end{array}}
  \!\!\!\!\!\!
  \epsilon(\tau)
  (x_2 x_1- x_1 x_2) x_{\tau(3)} \cdots x_{\tau(n+1)}
  \\
  &\quad
  -
  \!\!\!\!\!\!
  \sum_{\tiny \begin{array}{c} \tau \in S_{n+1} \\ \tau(n) = 2 \\ \tau(n{+}1)=1 \end{array}}
  \!\!\!\!\!\!
  {\epsilon(\tau)}
  x_{\tau(1)} \cdots x_{\tau(n-1)} (x_2 x_1 - x_1 x_2)
  +
  \beta (x_2x_1- x_1 x_2).
  \end{align*}
Adding the last three expressions for $A$, $B$ and $C$ gives
  \[
  M
  =
  \sum_{i=1}^{n+1}
  \!\!\!\!
  \sum_{\tiny \begin{array}{c} \tau \in V_{n+1} \\ \tau(n{+}1) = i \end{array}}
  \!\!\!\!\!\!
  \epsilon(\tau)
  x_{\tau(1)} \cdots x_{\tau(n)} x_i
  -
  \sum^{n+1}_{i =1}
  \!\!\!\!
  \sum_{\tiny \begin{array}{c} \tau \in V_{n+1} \\ \tau(1) = i \end{array}} \!\!\!\!\!\!
  \epsilon(\tau)
  x_i x_{\tau(2)} \cdots x_{\tau(n+1)}.
  \]
The argument at the end of the proof of Theorem \ref{Nff} now shows that $M =
0$.
\end{proof}

\begin{corollary}
Let $n \ge 4$ be even and let $L$ be any non-simple $(n+1)$-dimensional $n$-Lie
algebra over $F$. In cases (0), (1a), (2a) and ($r$) of Theorem \ref{classf}, a
basis of the universal associative envelope $U(L)$ consists of the monomials
which do not contain any factor of the form
  \[
  x_{i_1} x_{i_2} \cdots x_{i_n}
  \;
  (i_1 > i_2 > \cdots > i_n).
  \]
In case (1b) of Theorem \ref{classf}, a basis of the universal associative
envelope $U(L)$ consists of the monomials which do not contain any factor of
the form
  \[
  x_{n+1}x_{1}
  \qquad \text{or} \qquad
  x_{i_1} x_{i_2} \cdots x_{i_n}
  \;
  (i_1 > i_2 > \cdots > i_n).
  \]
In case (2b) of Theorem \ref{classf}, a basis of the universal associative
envelope $U(L)$ consists of the monomials which do not contain any factor of
the form
  \begin{align*}
  &
  x_2 x_1,
  \qquad
  x_{n+1} x_n \dots x_2,
  \qquad
  x_{n+1}x_n\dots x_3 x_1,
  \qquad
  \\
  &
  x_{n+1} \cdots x_5 x_2 x_4 x_1
  \qquad \text{or} \qquad
  x_{n+1} \cdots \widehat{x}_i \cdots x_4 x_2 x_3 x_1 \; (i \ge 4).
  \end{align*}
Hence in every case $U(L)$ is infinite-dimensional.
\end{corollary}

\begin{corollary}
Let $n \ge 4$ be even and let $F$ be a field of characteristic $\ne 2$. For any
non-simple $(n{+}1)$-dimensional $n$-Lie algebra $L$ the natural map $i\colon L
\to U(L)$ is injective.
\end{corollary}


\section{Computational results for $n$ odd} \label{oddsection}

In this section we present computational results to illustrate the complexity
of finding a Gr\"obner basis for the ideal $I = \langle G_1, \dots, G_{n+1}
\rangle \subseteq F\langle x_1, \dots, x_{n+1} \rangle$ when $n$ is odd. These
computations were done with the computer algebra system \texttt{Maple}.

We consider the 4-dimensional simple 3-Lie algebra $L_4$; to clarify the
notation in this special case, we write $a, b, c, d$ in place of $x_1, x_2,
x_3, x_4$ for the basis elements. The structure constants are then as follows:
  \[
  [a,b,c] = d, \qquad
  [a,b,d] = -c, \qquad
  [a,c,d] = b, \qquad
  [b,c,d] = -a.
  \]
The original set of ideal generators, which is already self-reduced, is as
follows:
  \allowdisplaybreaks
  \begin{align*}
  G_1
  &=
     dcb
  -  dbc
  -  cdb
  +  cbd
  +  bdc
  -  bcd
  -  a,
  \\
  G_2
  &=
     dca
  -  dac
  -  cda
  +  cad
  +  adc
  -  acd
  +  b,
  \\
  G_3
  &=
     dba
  -  dab
  -  bda
  +  bad
  +  adb
  -  abd
  -  c,
  \\
  G_4
  &=
     cba
  -  cab
  -  bca
  +  bac
  +  acb
  -  abc
  +  d.
  \end{align*}
Noncommutative polynomials will made monic and their terms will be listed in
reverse deglex order so that their leading monomials occur first; sets of
polynomials will be listed in reverse deglex order of their leading monomials.

\subsection{First iteration}

Lemma \ref{com} shows that there is only one composition among $G_1, \dots,
G_4$ and Theorem \ref{Nff} gives its normal form:
  \[
  G_1 a - d G_4
  \xrightarrow{\text{nf}}
  N
  =
  d^2 + c^2 + b^2 + a^2.
  \]
Since $N \ne 0$, we must add $N$ to the set of ideal generators and repeat the
process. The new set of generators, which is already self-reduced, is $\{ G_1,
\, G_2, \, G_3, \, G_4, \, N \}$.

\subsection{Second iteration}

We obtain three new compositions and compute their normal forms:
  \allowdisplaybreaks
  \begin{align*}
  N c b - d G_1
  &\xrightarrow{\text{nf}}
  P_1
  =
     dcdb
  -  dbdc
  -  cdbd
  +  c^3b
  -  c^2bc
  -  cbc^2
  -  cb^3
  -  caba
  +  bdcd
  \\
  &\qquad\qquad
  +  bc^3
  +  bcb^2
  +  b^2cb
  -  b^3c
  +  baca
  +  acab
  -  abac
  +2 da
  -2 ad,
  \\
  N c a - d G_2
  &\xrightarrow{\text{nf}}
  P_2
  =
     dcda
  -  dadc
  -  cdad
  +  c^3a
  -  c^2ac
  -  cac^2
  -  cab^2
  -  ca^3
  +  b^2ca
  \\
  &\qquad\qquad
  -  b^2ac
  +  adcd
  +  ac^3
  +  acb^2
  +  aca^2
  +  a^2ca
  -  a^3c
  -  db
  +  bd,
  \\
  N c b - d G_3
  &\xrightarrow{\text{nf}}
  P_3
  =
     dbda
  -  dadb
  +  cbca
  -  cacb
  -  bdad
  -  bcac
  +  b^3a
  -  b^2ab
  -  bab^2
  \\
  &\qquad\qquad
  -  ba^3
  +  adbd
  +  acbc
  +  ab^3
  +  aba^2
  +  a^2ba
  -  a^3b
  +2 dc
  -2 cd.
  \end{align*}
The new set of generators, which is already self-reduced, is
  \[
  \{ \, P_1, \, P_2, \, P_3, \, G_1, \, G_2, \, G_3, \, G_4, \, N \, \}.
  \]

\subsection{Third iteration}

We obtain five new compositions:
  \begin{alignat*}{3}
  &
  P_1 d a - d c P_3 \xrightarrow{\text{nf}} Q_1,
  &\qquad
  &
  N c d b - d P_1 \xrightarrow{\text{nf}} Q_2,
  &\qquad
  &
  N c d a - d P_2 \xrightarrow{\text{nf}} Q_3,
  \\
  &
  N b d a - d P_3 \xrightarrow{\text{nf}} Q_4,
  &\qquad
  &
  P_1 a - d c G_3 \xrightarrow{\text{nf}} Q_5.
  \end{alignat*}
These compositions have 34, 20, 20, 20, 23 terms respectively; their normal
forms have 178, 35, 33, 56, 6 terms respectively. The simplest new generator is
  \[
  Q_5
  =
     dc^2
  +  db^2
  +  da^2
  -  c^2d
  -  b^2d
  -  a^2d.
  \]
The leading monomials of the others are
  \[
  \mathrm{LM}(Q_1) = dc^2bca, \quad
  \mathrm{LM}(Q_2) = dc^3b, \quad
  \mathrm{LM}(Q_3) = dc^3a, \quad
  \mathrm{LM}(Q_4) = dbc^2a.
  \]
We add these new noncommutative polynomials to the set of generators and obtain
  \[
  \{ \,
  Q_1, \, Q_2, \, Q_3, \, Q_4, \, P_1, \, P_2, \, P_3, \, Q_5, \, G_1, \, G_2, \, G_3, \, G_4, \, N
  \, \}.
  \]
However, this set of the generators is not self-reduced: the leading monomials
of some generators are factors of monomials occurring in other generators.
After performing self-reduction, we find that $Q_2$ and $Q_3$ become 0, and
$Q_1$ and $Q_4$ respectively become $R_1$ and $R_2$ with 187 and 58 terms and
leading monomials $dbc^3a$ and $dbc^2a$. The new self-reduced set of generators
is
  \[
  \{ \,
  R_1, \, R_2, \, P_1, \, P_2, \, P_3, \, Q_5, \, G_1, \, G_2, \, G_3, \, G_4, \, N
  \, \}.
  \]

\subsection{Fourth iteration}

We obtain six new compositions:
  \[
  P_1 c^3 a - d c R_1, \quad
  P_1 c^2 a - d c R_2, \quad
  N b c^3 a - d R_1, \quad
  N b c^2 a - d R_2, \quad
  N c^2 - d Q_5.
  \]
These compositions have 203, 74, 189, 60, 8 terms respectively. This suggests
that the algorithm may not terminate and that the Gr\"obner basis obtained by
this process from the original set of ideal generators may in fact be infinite.


\section*{Acknowledgements}

This work forms part of the doctoral thesis of the second author, written under
the supervision of the first author. The first author was partially supported
by a Discovery Grant from NSERC, the Natural Sciences and Engineering Research
Council of Canada. The second author was supported by a University Graduate
Scholarship from the University of Saskatchewan.



\begin{thebibliography}{99}

\bibitem{BaggerLambert}
  \textsc{J. Bagger and N. Lambert}:
  Modeling multiple M2-branes.
  \emph{Phys. Rev. D}
  75 (2007), no.~4, 045020, 7 pages.
  MR2304429 (2008a:81170)

\bibitem{BaiSong}
  \textsc{R. Bai and G. Song}:
  The classification of six-dimensional 4-Lie algebras.
  \emph{J. Phys. A}
  42 (2009), no.~3, 035207, 17 pages.
  MR2525315 (2010f:17008)

\bibitem{Bergman}
  \textsc{G. M. Bergman}:
  The diamond lemma for ring theory.
  \emph{Adv. in Math.}
  29 (1978), no.~2, 178--218.
  MR0506890 (81b:16001)

\bibitem{CasasInsuaLadra}
  \textsc{J. M. Casas, M. A. Insua and M. Ladra}:
  Poincar\'e-Birkhoff-Witt theorem for Leibniz $n$-algebras.
  \emph{J. Symbolic Comput.}
  42 (2007), no.~11-12, 1052--1065.
  MR2368072 (2008k:17002)

\bibitem{deAzcarragaIzquierdo}
  \textsc{J. A. de Azc\'arraga and J. M. Izquierdo}:
  $n$-ary algebras: a review with applications.
  \emph{J. Phys. A: Math. Theor.}
  43 (2010) 293001 (117 pages).

\bibitem{deGraaf}
  \textsc{W. A. de Graaf}:
  \emph{ Lie Algebras: Theory and Algorithms}.
  North-Holland, Amsterdam, 2000.
  MR1743970 (2001j:17011)

\bibitem{Filippov}
  \textsc{V. T. Filippov}:
  $n$-Lie algebras.
  \emph{Sibirsk. Mat. Zh.}
  26 (1985), no.~6, 126--140.
  MR0816511 (87d:08019)

\bibitem{Gustavsson}
  \textsc{A. Gustavsson}:
  One-loop corrections to Bagger-Lambert theory.
  \emph{Nuclear Phys. B}
  807 (2009), no.~1-2, 315--333.
  MR2466670 (2009k:81166)

\bibitem{Humphreys}
  \textsc{J. E. Humphreys}:
  \emph{Introduction to Lie Algebras and Representation Theory.}
  Graduate Texts in Mathematics, Vol. 9.
  Springer-Verlag, New York-Berlin, 1972.
  MR0323842 (48 \#2197)

\bibitem{InsuaLadra}
  \textsc{M. A. Insua and M. Ladra}:
  Gr\"obner bases in universal enveloping algebras of Leibniz algebras.
  \emph{J. Symbolic Comput.}
  44 (2009), no.~5, 517--526.
  MR2499928 (2010a:17006)

\bibitem{Ling}
  \textsc{W. Ling}:
  \emph{On the Structure of $n$-Lie Algebras}.
  Dissertation, University-GHS-Siegen, 1993, 61 pages.

\bibitem{Pozhidaev}
  \textsc{A. P. Pojidaev}:
  Enveloping algebras of Filippov algebras.
  \emph{Comm. Algebra}
  31 (2003), no.~2, 883--900.
  MR1968930 (2004b:17008)

\end{thebibliography}
\end{document}